\documentclass{article}
\usepackage{amsthm}
\usepackage{amsmath}
\usepackage{amssymb}
\usepackage{graphicx}
\usepackage{enumitem}
\usepackage{url}
\makeatletter
\let\@fnsymbol\@arabic
\makeatother
\newtheorem{theorem}{Theorem}[section]
\newtheorem{lemma}[theorem]{Lemma}
\newtheorem{proposition}[theorem]{Proposition}
\newtheorem{corollary}[theorem]{Corollary}
\newtheorem*{parity}{Parity Lemma}
 \theoremstyle{definition}
\newtheorem{question}[theorem]{Question}
\newtheorem*{sdbd}{Standard description of a band decomposition or
        ribbon graph}
\def\jau#1{{\frenchspacing #1}} 
\def\jti#1{\textit{\frenchspacing #1}} 
\def\jvo#1{\textbf{#1}} 
\def\lline{\hbox to\hsize} 
\let\implies\Rightarrow

\let\cgr\equiv
\def\mR{{\mathbb{R}}}
\def\mZ{{\mathbb{Z}}}
\long\def\ignore#1{}

\let\Si\Sigma
\let\rD\Delta
\let\rP\Pi
\let\bar\overline

\def\mmod{{\fam0 mod\ }}
\def\deg{{\fam0 deg}}
\def\sep{Z}			
\def\sepsd{{\sep}^+}
\def\vs{z}
\def\hs#1{\zeta({#1})}
\def\hssd#1{\zeta^+({#1})}
\let\bdy\partial

\def\cobdy{\delta}
\def\bary{S}			
\def\baryd{\bary^*}
\def\baryv{u}
\def\rad{\bary_{02}}
\def\final#1{#1'}

\def\ug#1{\Gamma({#1})}

\def\ir{c}
\def\ib{a}
\def\iq{q}
\def\iR{C}
\def\iB{A}
\def\iQ{Q}
\def\fR{\final{\iR}}
\def\fr{\final{\ir}}
\def\fB{\final{\iB}}
\def\fb{\final{\ib}}

\def\fQ{\final{\iQ}}

\def\fv{\final{v}}
\def\ff{\final{f}}
\let\sd\oplus
\def\bd{N}
\def\rg{O}
\def\embg#1{\Theta({#1})}
\def\bunv{\bigcup_{v \in V(G)}}
\def\bune{\bigcup_{e \in E(G)}}
\def\bund{\bigcup_{e \in D}}

\def\bunf{\bigcup_{f \in F(G)}}
\def\buni{\bigcup_{i=1}^k}

\def\vk{y}
\def\ek{d}
\def\vkvg{\nu}
\def\ekcr{\chi}
\def\vkfg{\phi}
\def\vkeg{\varepsilon_G}
\def\vkeb{\varepsilon_{\bary}}
\def\ekeg{\varepsilon_G}
\def\wkwg{\omega}
\def\VPPV{\final{V_{\fam0 ppv}}}
\def\VDPF{\final{V_{\fam0 dpf}}}
\def\VTDR{\final{V_{\fam0 tdb}}}
\def\FPPF{\final{F_{\fam0 ppf}}}
\def\FDPV{\final{F_{\fam0 dpv}}}
\def\FTPR{\final{F_{\fam0 tpb}}}
\let\implies\Rightarrow

\let\iff\Leftrightarrow
\let\tofrom\leftrightarrow
\def\eg{g}
\def\vg{x}
\def\crg{L}
\def\pet{^\times}

\begin{document}

\title{{\bf Partial duality and closed 2-cell embeddings} \\
	\quad \\
	{\sl To Adrian Bondy on his 70th birthday}}

\author{%
  M. N. Ellingham%
	\thanks{Supported by National Security Agency grant
H98230-13-1-0233 and Simons Foundation award 245715.}
${}^{,3}$%
	\\
  Department of Mathematics, 1326 Stevenson Center\\
  Vanderbilt University, Nashville, Tennessee 37240, U.S.A.\\
  \texttt{mark.ellingham@vanderbilt.edu}%
 \and
  Xiaoya Zha%
	\thanks{Supported by National Security Agency grant
H98230-13-1-0216.}
 ${}^,$\footnote{The United States
Government is authorized to reproduce and distribute reprints
notwithstanding any copyright notation herein.}%
	\\
  Department of Mathematical Sciences, Box 34 \\
  Middle Tennessee State University \\
  Murfreesboro, Tennessee 37132, U.S.A. \\
	\texttt{xiaoya.zha@mtsu.edu}%
}
 \date{April 28, 2016; to appear in \textit{Journal of Combinatorics}}
\maketitle

 \begin{abstract}
 In 2009 Chmutov introduced the idea of partial duality for embeddings
of graphs in surfaces.  We discuss some alternative descriptions of
partial duality, which demonstrate the symmetry between vertices and
faces.  One is in terms of band decompositions, and the other is in
terms of the \emph{gem} (graph-encoded map) representation of an
embedding.
 We then use these to investigate when a partial dual is a \emph{closed
$2$-cell embedding}, in which every face is bounded by a cycle in the
graph.
 We obtain a necessary and sufficient condition for a partial dual to be
closed $2$-cell, and also a sufficient condition for no partial dual to
be closed $2$-cell.
 \end{abstract}

 \section{Introduction}\label{intro}

 In this paper a {\em surface\/} $\Si$ means a connected compact $2$-manifold
without boundary.
 By an {\em open\/} or {\em closed disk\/} in $\Si$ we mean a subset of
the surface homeomorphic to such a subset of $\mR^2$.  By a \emph{simple
closed curve} or \emph{circle} in $\Si$ we mean an image of a circle in
$\mR^2$ under a continuous injective map; a \emph{simple arc} is a
similar image of $[0,1]$.
 The closure of a set $S$ is denoted $\bar S$, and the boundary is
denoted $\bdy S$.

 All our graphs are finite and may (and frequently do) have multiple
edges or loops.
 An \emph{embedding} of a graph in a surface maps each vertex to a point
and each edge to a simple arc or simple closed curve, so that the only
pairwise intersections of those points and arcs correspond to the
incidences between an edge and its endvertices.  For convenience, we
usually do not distinguish between a vertex or edge of a graph and its
image in a graph embedding.
 A \emph{face} of a graph embedding is a connected component of the
topological subspace of the surface obtained by deleting the images of
all vertices and edges.
 If two embeddings of a graph are related by a homeomorphism of the
surface we consider them to be equivalent.
 We typically use $G$ to denote not just a graph, but a graph embedding,
with vertex set $V(G)$, edge set $E(G)$ and face set $F(G)$; $\ug{G}$
denotes the underlying graph.
 %
 %

 A graph embedding in a surface is \emph{$2$-cell}, \emph{open $2$-cell}
or \emph{cellular} if every face is an open disk.  A $2$-cell embedding
is a \emph{closed $2$-cell} embedding if every face is bounded by a
simple closed curve; equivalently, the boundary walk of each face is a
cycle in the graph.
 The graph is necessarily nonseparable, so if it has at least two
edges it has no loops and no cutvertices.
 (There is some ambiguity in the definition of `closed $2$-cell' in
the literature.  The definition we use is the one given by Barnette
\cite{Barn87} in what seems to be the first mention of the concept.  
 Another common definition is that the closure of every face is a closed
disk; this is not quite the same, because it allows the boundary of a
face, considered as a subgraph, to be a cycle with attached trees.  For
$2$-connected graphs the two definitions are equivalent.  Our graphs,
however, are not necessarily $2$-connected so it is important to be
clear about which definition we are using.)

 %

 Because all faces in a $2$-cell embedding are homeomorphic to an open
disk, we can describe such an embedding without giving topological
information for each face.  In fact, it is well known that $2$-cell
embeddings can be described in a completely combinatorial way, using a
\emph{rotation system} with \emph{edge types} or \emph{edge signatures}.
 See \cite[Section 3.2]{GT} and \cite[Sections 3.2, 3.3]{MT}. 
 A $2$-cell embedding can also be described in a number of other ways,
including as a \emph{band decomposition} or as a \emph{reduced band
decomposition}, also called a \emph{ribbon graph}.
 We assume the reader is familiar with these different representations
of graph embeddings.

 For the remainder of this paper all embeddings will be $2$-cell
embeddings of graphs in surfaces.
 We will sometimes mention this explicitly, but often just assume it
implicitly.
 Since our surfaces are connected, these are necessarily embeddings of
connected graphs.

 It is well known that for every $2$-cell graph embedding $G$ there is a
unique \emph{dual} graph embedding $G^*$ (sometimes called the
\emph{geometric dual}) in the same surface, obtained as follows.  Insert
one vertex of $G^*$ in each face of $G$, and for each edge $e$ of $G$
add an edge $e^*$ in $G^*$, crossing $e$, which joins the vertices of
$G^*$ corresponding to the faces on either side of $e$.  Then $G^*$ is
also $2$-cell, the faces of $G^*$ correspond to the vertices of $G$ and
vice versa, and $(G^*)^* = G$.  Since there is a bijection $e \tofrom
e^*$, $e$ and $e^*$ are often thought of as the same object.  We will
distinguish between them only when necessary (usually when considering
them as actual curves in a surface).
 %
 %

 There are other duals for graph embeddings.  For example, the
\emph{Petrie dual} is formed by toggling the signature of every edge
(from $+1$ to $-1$, or vice versa) in a representation of the embedding
in terms of rotation systems and edge signatures.  It is clear that it
is possible to do this in a partial fashion, by toggling the signatures
for some subset of edges, instead of for all edges.  Taking the partial
Petrie dual is a useful operation.
 For example, the process used by the authors in \cite{EZ11} to convert
projective-planar embeddings into orientable closed $2$-cell embeddings,
although described in that paper as a multistep process of inserting
crosscaps, then cutting through them and capping off, is really just
taking partial Petrie duals.

 No one, however, seems to have considered a partial version of
geometric duality for graph embeddings, until Chmutov \cite{Chmu09} in
2009.
 He was motivated the desire to unify several Thistlethwaite-type
theorems; results of this type relate polynomials for knots or links to
graph polynomials.
 He defines the partial dual, with respect to a subset of edges, for a
signed graph, where each edge is assigned either $+$ or $-$.
 His main result is an invariance formula for a certain specialization
of the Bollob\'{a}s-Riordan polynomial of signed graphs under partial
duality.
 The Bollob\'{a}s-Riordan polynomial is formulated in terms of ribbon
graphs \cite{BR02}.
 Chmutov's original definition of partial duality therefore used ribbon
graphs, and another related representation for $2$-cell embeddings,
which he calls an \emph{arrow presentation}.

 Our interest is purely in the topological structure of a graph
embedding under the partial duality operation, and we therefore ignore
	the part of Chmutov's operation that involves changing the signs of
edges, and just deal with unsigned graphs.
 Several alternative constructions for Chmutov's
partial duality operation have appeared in the literature, and are
summarized in \cite[Section 2.2]{EM}.
 All of these are based on ribbon graphs and arrow presentations.  We
will provide details of one of these constructions in Section
\ref{pd-gem-surf}.

 Section \ref{pd-gem-surf} discusses two alternative representations of
partial duality that do not use ribbon graphs or arrow presentations.
 The first uses the band decomposition representation of an embedding.
 The second uses a completely combinatorial representation of a graph
embedding as a \emph{graph-encoded map}, or \emph{gem}.
 Both of these definitions give explicit descriptions of the faces as
well as the vertices of the partial dual, and demonstrate symmetry with
respect to vertices and faces.

 In Section \ref{c2c} we consider the question of whether a partial dual
is closed $2$-cell.  
 We obtain a necessary and sufficient condition for a partial dual
of a general graph embedding to be closed $2$-cell.  This has a
particularly simple form for initial embeddings that are closed
$2$-cell and satisfy certain additional properties.
 We also obtain a sufficient condition for no partial dual to be closed
$2$-cell.

 Before we begin we offer an apology of sorts to the reader.  Much of
our work is just manipulating definitions, and converting ideas from one
representation of an embedding into another.   We work with many
different ways of looking at an embedding, and many related but not
identical concepts, and the reader may find this confusing.
 Unfortunately, this seems to be necessary to get anywhere in this area:
different representations are useful for different purposes.
 We have tried to keep the notation as consistent as possible, to help
the reader follow what is going on.

 \section{Partial duality for band decompositions and gems}
 \label{pd-gem-surf}

 In this section we discuss two representations of partial duality. 
 Graph embeddings can be described in many ways, and so there are also
many ways to describe partial duality.
 The motivation of the two descriptions we give here is that they enable
our work in Section \ref{c2c}, characterizing closed $2$-cell partial
duals.
 The description in terms of band decompositions (Proposition
\ref{bd-pd}) is convenient in situations where we deal with an actual
embedding, rather than a model such as a ribbon graph.
 We specialize this to a particular type of band decomposition which we
call a \emph{corner graph}.
 The description in terms of gems (Proposition \ref{gem-pd}) is very
simple.
 We provide examples which we hope will be helpful for understanding.

 We briefly remind the reader how band decompositions and ribbon graphs
of a graph embedding $G$ are found; for a complete description see
\cite[Subsection 3.2.1]{GT}.
 For the \emph{band decomposition} we `fatten' both vertices and edges
into closed disks, known as \emph{$0$-bands} for the vertices, and
\emph{$1$-bands} for the edges; each face then also corresponds to a
closed disk, called a \emph{$2$-band}.
 A \emph{reduced band decomposition} or \emph{ribbon graph}
representation may then be obtained by discarding the $2$-bands.

 For adjacent edges of $G$, their $1$-bands may or may not intersect, at
a point.  We can modify a band decomposition so that they always
intersect, or always do not intersect.
 To allow adjacent $1$-bands to intersect when they are the same band,
we loosen the definition of a $1$-band a little by allowing it to be a
closed disk with two points on its boundary identified, or a closed disk
with two pairs of points $p_1$ and $p_2$, $p_3$ and $p_4$ identified,
where $p_1, p_2, p_3, p_4$ appear in that order around the boundary.

 \begin{sdbd}
 We will describe a band decomposition or ribbon graph of a graph
embedding $G$ in the following standard way.  The $0$-bands will be
denoted $\bar{\ir_v}$ for $v \in V(G)$; $\ir_v$ is the interior of the band.
 Similarly the $1$-bands will be denoted $\bar{\iq_e}$ for $e \in E(G)$,
and the $2$-bands (if we have them) will be denoted $\bar{\ib_f}$ for $f
\in F(G)$.  The boundaries of the bands in a band decomposition $\bd$ of
a connected graph embedding with at least one edge may be regarded as an
embedded graph $\embg{\bd}$: the vertices are points where three or four
bands (one $0$-band, one $2$-band, and one or two $1$-bands) meet, the
boundary segments between these points define the edges, and the
interiors of the bands are the faces.  In $\embg{\bd}$, every face
$\ir_v$ is bounded by a cycle $\iR_v$, every face $\ib_f$ is bounded by
a cycle $\iB_f$, and every face $\iq_e$ is bounded by a closed trail (no
repeated edges) $\iQ_e$ of length four, which is a cycle except possibly
when $e$ is incident in $G$ with a vertex or face of degree one.  If
$\rg$ is a ribbon graph we form a band decomposition $\bd$ by gluing a
$2$-band along every boundary component, and then define
$\embg{\rg}=\embg{\bd}$.

 We let $\iR = \bunv \iR_v$ and $\iB = \bunf \iB_f$.
 For any $D \subseteq E(G)$ we let $\iQ_D = \bund \iQ_e$, $\iR_D = \iR
\cap \iQ_D$ and $\iB_D = \iB \cap \iQ_D$.
 The edges of each $\iQ_e$ alternate between edges of $\iR$ and of $\iB$.

 The following may help the reader remember the notation above.
 We can think of colouring the edges of $\embg{\bd}$, generalizing the
standard colouring of edges in a graph-encoded map, which we discuss
later in this section.  An edge between $0$- and $1$-bands is red, an
edge between $0$- and $2$-bands is yellow, and an edge between $1$- and
$2$-bands is blue.
 We use this colour coding consistently in our figures.
 The boundary of a $0$-band contains red edges and possibly
some yellow edges, so we use the letter $C$ for \emph{crimson}, a type of red. 
 The boundary of a $2$-band contains blue edges and possibly some yellow
edges, so we use the letter $A$ for \emph{azure}, a type of blue.  (We
cannot use $R$ and $B$ because we use those later for edges in a
graph-encoded map.)
 The boundary of a $1$-band is a quadrilateral (with alternating red and
blue edges), so we use the letter $Q$.
 \end{sdbd}

 The following description of partial duals is convenient for us because
it uses only ribbon graphs; it does not require arrow presentations.

 \begin{proposition}[Bradford, Butler and Chmutov {\cite[Definition
1.2]{BBC}}]
 \label{rg-pd}
 Given a connected graph embedding $G$ represented by a ribbon graph
$\rg$, and $D \subseteq E(G)$, we may construct a ribbon graph $\rg^D$
corresponding to $G^D$, the partial dual of $G$ with respect to $D$, as
follows.

 \smallskip
 \noindent \textnormal{(1)}
 Take the (possibly disconnected) ribbon subgraph
 $\rD$ of $\rg$ consisting of bands $\bar{\ir_v}$ for $v \in V(G)$ and
$\bar{\iq_e}$ for $e \in D$ (using our standard description above).
 Let the boundary of $\rD$ be $\fR_1 \cup \fR_2 \cup \ldots \cup \fR_k$,
where each $\fR_i$ is a simple closed curve, representing a cycle in
$\embg{\rg}$.

 \smallskip
 \noindent \textnormal{(2)}
 Add to $\rg$ a new $0$-band $\bar{\fr_i}$ with boundary
$\fR_i$ for $i = 1, 2, \ldots, k$.
 Then remove every $0$-band interior $\ir_v$ for $v \in V(G)$ to give
$\rg^D$.
 \end{proposition}

 For our purposes the difficulty with this description of partial
duality, or other descriptions involving ribbon graphs and arrow
presentations, is that it does not explicitly describe the faces of the
partial dual.
 The vertices of $G^D$ are explicitly determined by
$\fR_1 \cup \fR_2 \cup \ldots \fR_k$, but the faces are determined only
implicitly.
 For our results on closed $2$-cell embeddings in Section \ref{c2c} we
need an explicit description of both vertices and faces.
 The problem is that ribbon graphs themselves do not explicitly describe
faces, since we discard the $2$-bands.

 We use $S \sd T$ to denote the symmetric difference of two sets,
$(S-T) \cup (T-S)$.

 \begin{proposition}
 \label{bd-pd}
 Given a connected graph embedding $G$ in $\Si$ represented by a band
decomposition $\bd$ of $\Si$, and $D \subseteq E(G)$, we may construct a
band decomposition $\bd^D$ corresponding to $G^D$, the partial dual of
$G$ with respect to $D$, as follows, using our standard decomposition
for $\bd$.

 \smallskip
 \noindent \textnormal{(1)}
 Take the (possibly disconnected) subspace
$\rD$ of $\Si$ that is the union of $\bar{\ir_v}$ for all $v \in V(G)$
and $\bar{\iq_e}$ for all $e \in D$.
 Let the boundary of $\rD$ be $\fR = \fR_1 \cup \fR_2 \cup \ldots \cup \fR_k$,
where each $\fR_i$ is a simple closed curve, representing a cycle in
$\embg{\bd}$.
 %

 \smallskip
 \noindent \textnormal{(2)}
 Let $P = E(G)-D$ and take the (possibly disconnected) subspace
 $\rP$ of $\Si$ that is the union of $\bar{\ir_v}$ for all $v \in V(G)$
and $\bar{\iq_e}$ for all $e \in P$.
 Let the boundary of $\rP$ be $\fB = \fB_1 \cup \fB_2 \cup \ldots \cup
\fB_\ell$, where each $\fB_j$ is a simple closed curve, representing a
cycle in $\embg{\bd}$.

 \smallskip
 \noindent \textnormal{(3)}
 Delete from $\bd$ all $0$-band interiors $\ir_v$ for $v \in V(G)$
and all $2$-band interiors $\ib_f$ for $f \in F(G)$, leaving
 the union of $\embg{\bd}$ and all $\bar{\iq_e}$ for $e \in E(G)$.
 Glue on new $0$-bands $\bar{\fr_i}$ with $\bdy \fr_i = \fR_i$ for $i =
1, 2, \ldots, k$, and new $2$-bands $\bar{\fb_j}$ with $\bdy \fb_j =
\fB_j$ for $j = 1, 2, \ldots, \ell$.
 The result is $\bd^D$.
 \end{proposition}

 \begin{proof}
 Form a ribbon graph $\rg$ by discarding the $2$-bands of $\bd$.
 We can obtain a band decomposition of $G^D$ by adding $2$-bands along
the boundary of $\rg^D$ from Proposition \ref{rg-pd}.  The $0$-bands and
$1$-bands of $\bd^D$ as constructed here are identical to those for
$\rg^D$ from Proposition \ref{rg-pd}, so we just need to show that the
boundaries of the $2$-bands of $\bd^D$, which form $\bdy \rP$, are the
same as the boundary of $\rg^D$, i.e., that $\bdy \rP = \bdy (\rg^D)$.

 To avoid cumbersome notation, all equations in this paragraph are
correct up to adding or deleting finitely many points (which will be
vertices of $\embg{\bd}$); this is sufficient to prove our conclusion.
 Since
 $\rD = \left( \bunv \bar{\ir_v} \right)
	\cup \left( \bund \bar{\iq_e} \right)$
 we have
 $\bdy \rD =  \left( \bunv \iR_v \right)
	\sd \left( \bund \iQ_e \right)
   = \iR \sd \iQ_D$.
 Similarly, $\bdy \rP = \iR \sd \iQ_P$.
 Now $\rg^D = \left( \buni \bar{\fr_i} \right)
	\cup \left( \bune \bar{\iq_e} \right)$
 so that
 $\bdy (\rg^D) = \left( \buni \fR_i \right)
	\sd \left( \bune \iQ_e \right)
   = \bdy \rD \sd \iQ = (\iR \sd \iQ_D) \sd \iQ = \iR \sd (\iQ \sd \iQ_D)
   = \iR \sd \iQ_P = \bdy \rP$, as required.
 \end{proof}

 As pointed out by one of the referees, the claim that $\bdy \rP =
\bdy(O^D)$ in the above proof can also be proved using basic
properties of partial duality for ribbon graphs established by Chmutov
\cite{Chmu09}.

 When we are forming a dual $G^D$ we refer to the elements of $D$ as
\emph{dual} edges and the elements of $E(G)-D$ as \emph{primal} edges.
 In Proposition \ref{bd-pd}, loosely $\rD$ is the region of the surface
$\Si$ containing all vertices and the dual edges, and its boundary $\fR
= \bdy \rD$ specifies the new vertices.
 Similarly, $\rP$ is the region of the surface containing all vertices
and the primal edges, and its boundary $\fB = \bdy \rP$
specifies the new faces.
 We identify the new vertex associated with $\fR_i$ as $\fv_i$, and the
new face associated with $\fB_j$ as $\ff_j$.

 The construction of Proposition \ref{bd-pd} explicitly describes the
faces (via $2$-bands), as well as the vertices (via $0$-bands), of
$G^D$.  This construction also shows that partial duality treats dual and
primal edges in some sense symmetrically.
  Moreover, the union of the bands not in $\rP$, namely $\bar{\ib_f}$
for all $f \in F(G)$, and $\bar{\iq_e}$ for all $e \in D$, forms a
region $\rD^*$ complementary to $\rP$ (ignoring overlap along the
boundary), so $\bdy \rD^* = \bdy \rP$.
 Therefore, the above description can also be expressed in a way that
treats faces and vertices symmetrically, with $\fR = \bdy \rD$ and $\fB
= \bdy \rD^*$.


 \smallskip
 It is helpful to remove the ambiguity in the definition of band
decomposition as to whether $1$-bands can intersect.  For our main
result in Section \ref{c2c}, we will use
a band decomposition where $1$-bands corresponding to adjacent edges
always intersect.  To help familiarize the reader with this, we use it
as the basis for our example to illustrate Proposition \ref{bd-pd}.

 We first construct the \emph{barycentric subdivision} $\bary$ of a
graph embedding $G$ (see \cite[p. 154]{MT}), which we will use again
later.
 We replace each edge $e$ by a path of length two through a new vertex
$\baryv_e$, then add a new vertex $\baryv_f$ in each face $f$, joined to
all vertices (including new vertices $\baryv_e$) on the boundary of $f$.
 Let $U_0 = V(G)$,
 $U_1 = \{\baryv_e \;|\; e \in E(G)\}$ and 
 $U_2 = \{\baryv_f \;|\; f \in F(G)\}$.
 Then $\bary$ is a triangulation, where each triangle contains one vertex
from each of $U_0$, $U_1$ and $U_2$.
 Let $E_{jk}$, $0 \le j < k \le 2$, be the set of edges between $U_j$
and $U_k$ in $\bary$.
 The \emph{radial graph} $\rad$ is the subgraph of $\bary$ induced by
$U_0 \cup U_2$; its edge set is $E_{02}$.
 Define a \emph{corner} to be the midpoint of an edge of $E_{02}$.
 Every corner $\vk$ is associated with a vertex $\vkvg(\vk)$ of $G$, a
face $\vkfg(\vk)$ of $G$, and with an unordered pair $\vkeg(\vk)$ of
(possibly equal) edges of $G$.  It is also associated with an edge
$\vkeb(\vk)$ of $\bary$.  A corner can be thought of as marking the
place where a vertex, a face and two edges of $G$ come together.

 The \emph{corner graph} $K$ is a graph embedding with the corners of
$G$ as its vertices.
 For edges we add a cycle $\iR_v$ of length $\deg_G(v)$ around each vertex
$v$ of $G$, through the corners of $v$ in their order around $v$, and a
cycle $\iB_f$ of length $\deg_G(f)$ inside each face $f$ of $G$, through
the corners of $f$ in their order around $f$.
 Then $K$ has three types of face:
 each $v \in V(G)$ is in a face $\ir_v$ with $\bdy \ir_v = \iR_v$,
 each $f \in F(G)$ contains a face $\ib_f$ with $\bdy \ib_f = \iB_f$,
 and each $e \in E(G)$ with endvertices $u$ and $v$ intersects a unique
face $\iq_e$ of degree $4$ in addition to $\ir_u \cup \ir_v$.
 Each vertex of $K$ has degree $4$, $\iR = \bunv \iR_v$ and $\iB = \bunf \iB_f$
are edge-disjoint $2$-factors of $K$ with $\iR \cup \iB = K$, and $\{\iQ_e
\;|\; e \in E(G)\}$ is a decomposition of $K$ into edge-disjoint closed
trails of length $4$ (so that $\iQ = \bune \iQ_e = K$).  The closures of the
faces of $K$ give a band decomposition of $G$.

 Each edge $\ek$ of $K$ is associated with a unique edge $\ekeg(\ek)=e$
of $G$ by $\ek \in E(\iQ_e)$, so $\ek$ either parallels or crosses $e$.
 We can also define a function $\ekcr(\ek)$ for $\ek \in E(K)$, which
takes the value $1$ if $\ek$ crosses $\ekeg(\ek)$ in $\Si$, and $0$
when $\ek$ parallels $\ekeg(\ek)$.
 In other words, $\ekcr(\ek) = 1$ if $\ek \in E(\iR)$, and $0$ if $\ek
\in E(\iB)$.

 If we dualize a set of edges $D$ in $G$, then from Proposition
\ref{bd-pd}, the corner graph $K^D$ corresponding to the partial dual
$G^D$ has the same underlying graph as $K$.  We also keep the same
$1$-bands, but get new $0$-bands and $2$-bands using Proposition
\ref{bd-pd}.

 \begin{figure}[tb]
 \centering \scalebox{1}{\includegraphics{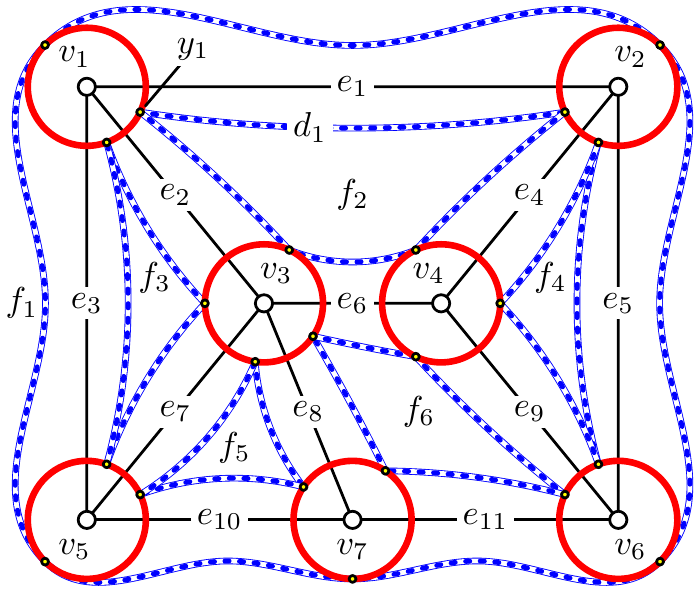}}
 \caption{Graph embedding and its corner graph\label{cnr-orig}}
 \end{figure}

 We illustrate this using Figure \ref{cnr-orig}, where we see a plane
graph $G$ (thin lines) and its corner graph $K$, with edges of $\iR$
shown in red (solid) and edges of $\iB$ shown in blue (dashed).
 Each face of $K$ is labelled by a vertex, edge or face of $G$.
 For the labelled vertex $\vk_1$ of $K$, we have $\vkvg(\vk_1) = v_1$,
$\vkfg(\vk_1) = f_2$, and $\vkeg(\vk_1) = \{e_1, e_2\}$.
 For the labelled edge $\ek_1$ of $K$, we have $\ekeg(\ek_1) = e_1$, and
$\ekcr(\ek_1) = 0$, since $\ek_1$ does not cross $e_1$.

 Now suppose we form $G^D$ where $D = \{e_2, e_3, e_4, e_7, e_{10}\}$;
edges of $D$ are shown as wavy edges in Figure \ref{cnr-pd}.  We now
change the colours of the edges of $K$ corresponding to $D$ to describe
the new vertices and faces of $K^D$.  The boundary components of $\bdy
\rD$, providing new vertices, are shown in red (solid); these form $\fR
= \fR_1 \cup \fR_2 \cup \fR_3 \cup \fR_4$, enclosing regions containing
the edges of $D$.
 The boundary components of $\bdy \rP$, providing new faces, are shown
in blue (dashed); these form $\fB = \fB_1 \cup \fB_2 \cup \fB_3$,
enclosing regions containing the edges of $P = \{e_1, e_5, e_6, e_8,
e_9, e_{11}\}$.

 \begin{figure}[tb]
 \centering \scalebox{1}{\includegraphics{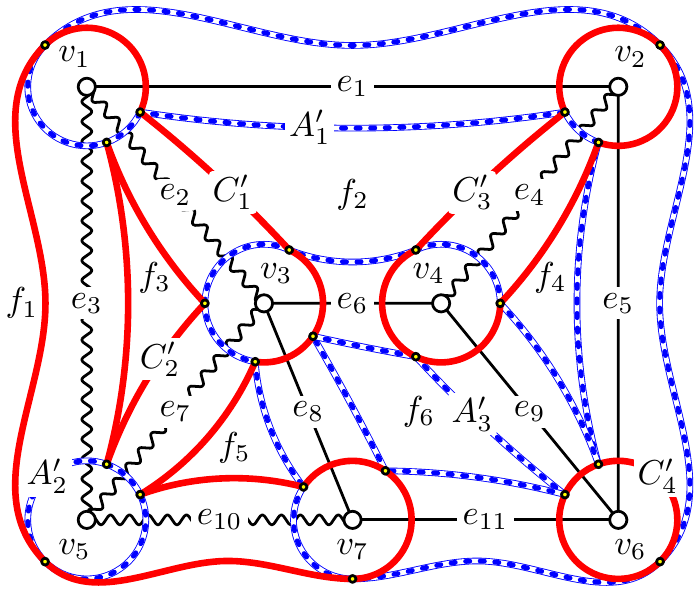}}
 \caption{Dualized corner graph\label{cnr-pd}}
 \end{figure}

 The edges incident with a new vertex or face, and their order, are
determined by applying $\ekeg$ to the edges of the appropriate cycle
$\fR_i$ or $\fB_j$.  For example, new vertex $\fv_3$ corresponds to
$\fR_3$, and tracing around $\fR_3$ clockwise we see that we parallel or
cross edges $e_1$, $e_5$, $e_4$, $e_9$, $e_6$ and $e_4$ again.
 That is the order of edges around $\fv_3$.
 The new face $\ff_2$ corresponds to $\fB_2$, and tracing around that we
see that we cross edges $e_3$, $e_7$ and $e_{10}$.   That is the order
of edges around $\ff_2$.
 Chmutov \cite{Chmu09} showed that a partial dual of an orientable
embedding is always orientable.  So knowing the rotation around each
vertex, we can construct the partial dual embedding.
 In the nonorientable case, we must be more careful; we need to know how
the edges of a cycle $\iQ_e$, corresponding to edges, are traversed by
the corresponding new vertices, so we can decide whether the
corresponding edge in the dual is twisted (signature $-1$) or not.  That
is the information carried by the arrow presentation in Chmutov's
original formulation of partial duality; in Propositions \ref{rg-pd} and
\ref{bd-pd} it is carried by the faces $\bar{\iq_e}$, which we never
delete.



 \smallskip
 We also consider a band decomposition where all $1$-bands are pairwise
disjoint.  We again base it on the barycentric subdivision $\bary$,
using the notation developed above.
 Since $\bary$ is a triangulation, its dual $\baryd$ is a cubic graph
embedding where each vertex is incident with exactly one edge from each
of $E_{01}^*$, $E_{02}^*$ and $E_{12}^*$, where $E_{jk}^*$ is the set of
dual edges for $E_{jk}$ (here we distinguish between edges and their
duals).
 If $U_j^*$ denotes the set of faces of $\baryd$ dual to the vertices in
$U_j$, then the closure of each element of $U_j^*$ is a $j$-band in a band
decomposition of $G$ in which $1$-bands are pairwise disjoint, with
 $\iR = \bigcup_{\ir \in U_0^*} \bdy \ir$,
 $\iQ = \bigcup_{\iq \in U_1^*} \bdy \iq$ and
 $\iB = \bigcup_{\ib \in U_2^*} \bdy \ib$.

 Colour the edges of $\baryd$ so that edges of $E_{01}^*$ are red, of
$E_{02}^*$ are yellow, and of $E_{12}^*$ are blue.
 (Our choice of colours here follows \cite{BL} and agrees with that in
our standard description of a band decomposition, earlier.)
 This is a proper $3$-edge-colouring.
 In an edge-coloured graph, we define a \emph{bigon} to be a cycle whose
edges alternate between two colours.
 Each face $\ir_v \in U_0^*$ is bounded by a red-yellow bigon; each
$\iq_e \in U_1^*$ is bounded by a red-blue bigon; and each $\ib_f \in
U_2^*$ is bounded by a blue-yellow bigon.
 If we take the underlying graph $\ug{\baryd}$ and keep the edge
colours, we have the \emph{graph-encoded map} or \emph{gem}
representation $M$ of $G$.
 We can clearly recover the graph embedding $\baryd$, and hence $G$, by
gluing a disk onto each bigon of $M$.

 Although we obtained $M$ from an embedding, it is just an abstract
connected, properly $3$-edge-coloured, cubic graph.  The red-yellow
bigons of $M$ correspond to vertices of $G$, the red-blue bigons
correspond to edges of $G$ and are all $4$-cycles, and the
blue-yellow bigons correspond to faces of $G$.  Note that $M$ may have
parallel edges in red-yellow or blue-yellow pairs, which correspond to
vertices and faces of degree one, respectively, but $M$ has no parallel
edges in other colour combinations (red-blue, or two of the same
colour), and $M$ has no loops.

 Gems are a purely combinatorial representation of an embedding.
 Any properly $3$-edge-coloured cubic graph $M$ with colours red, yellow
and blue in which the red-blue bigons are $4$-cycles is the gem of some
embedding.
 The idea of representing an embedding by a $3$-edge-coloured cubic
graph appears to be due to Robertson \cite[Section 4]{Robe71}.
 It was developed (appropriately for us) as part of a general theory of
graph embedding duality by Lins \cite{Lins82}, and used as the basis of
a general theory of graph embeddings by Bonnington and Little \cite{BL}.
 It has been extended to more general $2$-dimensional embeddings, and to
representations of higher dimensional objects; see \cite{ChVTpre}.

 We now show that partial duality can be expressed very simply in terms
of gems.  This is a result we have presented in the past (see for
example \cite{Elli12}) but this is the first proof we have given in
print.
 Recently Chmutov and Vignes-Tourneret developed a theory of partial
duality for hypermaps, which generalize graph embeddings.  They have a
generalization of Proposition \ref{gem-pd} in that context
\cite[Theorem 3.11]{ChVTpre}.

 \begin{proposition}
 \label{gem-pd}
 Suppose $G$ is a $2$-cell embedding of a connected graph $G$, with gem
$M$.  Let $D \subseteq E(G)$.  Then the partial dual $G^D$ corresponds
to a gem $M^D$ obtained by swapping the colours red and blue on each
bigon of $M$ corresponding to an edge in $D$.
 \end{proposition}

 \begin{proof}
 Everything we write here involves sets of edges, but for brevity we
just write $H$ instead of $E(H)$, for any subgraph $H$ of $M$.

 Let $R = \iR \cap \iQ$, $Y = \iR \cap \iB$ and $B = \iB \cap \iQ$ be
the sets of red, yellow and blue edges in $M$, respectively.
 Let $R_D = R \cap Q_D$, $R_P = R \cap Q_P$, $B_D = B \cap Q_D$ and $B_P
= B \cap Q_P$.  Then $R_D$, $R_P$, $B_D$, $B_P$ and $Y$ are disjoint and
their union is $E(M)$.
 We have $\iR = R_D \cup R_P \cup Y$, $\iB = B_D \cup B_P \cup Y$, and
$\iQ = R_D \cup R_P \cup B_D \cup B_P$.

 Suppose we take the partial dual; we use results and notation from
Proposition \ref{bd-pd} and its proof.
 We have $\fR = \iR \sd \iQ_D = (R_D \cup R_P \cup Y) \sd (R_D \cup B_D)
= R_P \cup B_D \cup Y$.
 Similarly, $\fB = \iR \sd \iQ_P = R_D \cup B_P \cup Y$.
 And $\fQ = \iQ = R_D \cup R_P \cup B_D \cup B_P$.
 Therefore, for $M^D$, $\final{R} = \fR \cap \fQ = R_P \cup B_D$,
$\final{Y} = \fR \cap \fB = Y$, and $\final{B} = \fB \cap \fQ = R_D \cap
B_P$.  In other words, we exchange red and blue on edges in $Q_D$, and
leave other colours alone, as required.
 \end{proof}

 The description of partial duality in Proposition \ref{gem-pd} treats
red and blue edges, and hence red-yellow and blue-yellow bigons,
symmetrically, again illustrating that partial duality treats vertices
and faces symmetrically.

 We provide an example in Figures \ref{gem-pic} and \ref{gem-vf}.  At
left in Figure \ref{gem-pic} we see a small planar graph and its gem
 (with colours indicated as red = solid, blue = dashed, and yellow =
hollow). Suppose we dualize $D = \{e_1, e_5\}$.  At right in Figure
\ref{gem-pic} we see the original graph, with the dual edges shown as
wavy edges, and the new gem.  In Figure \ref{gem-vf} we have separated
out the red-yellow and blue-yellow bigons in the partial dual, and we
see that there are now three vertices and only one face.

 \begin{figure}[tb]
  \lline{%
	\hfill
	\scalebox{1}{\includegraphics{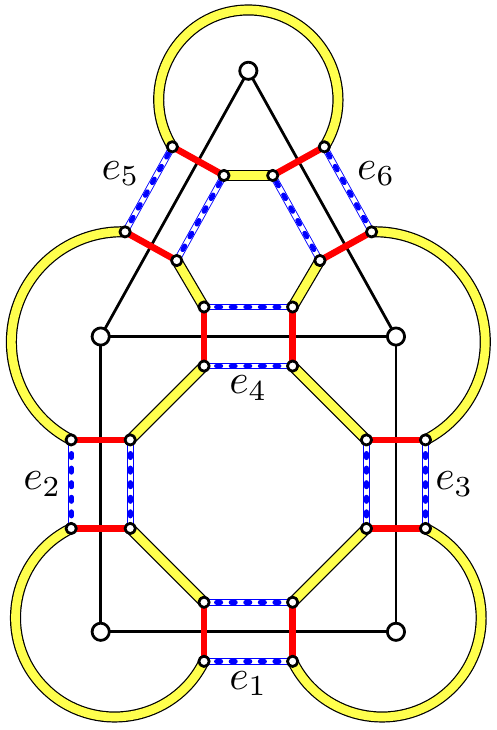}}%
	\kern2cm
	\scalebox{1}{\includegraphics{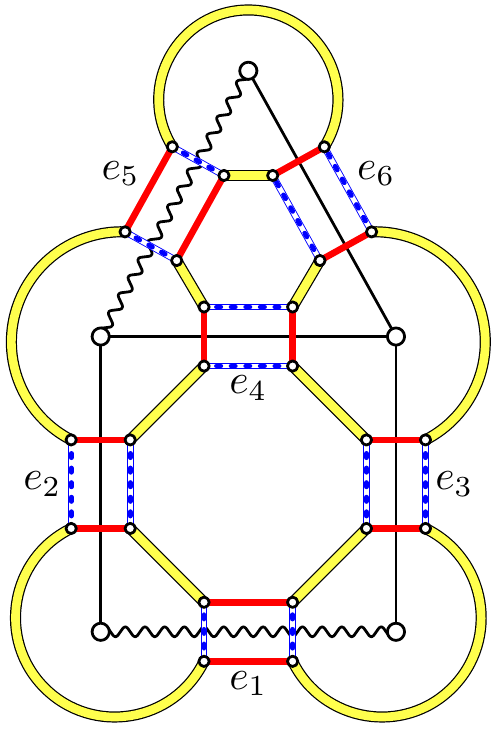}}%
	\hfill
  }
 \caption{Left: graph and gem; right: after partial dual\label{gem-pic}}
 \end{figure}

 \begin{figure}[tb]
  \lline{%
	\hfill
	\scalebox{1}{\includegraphics{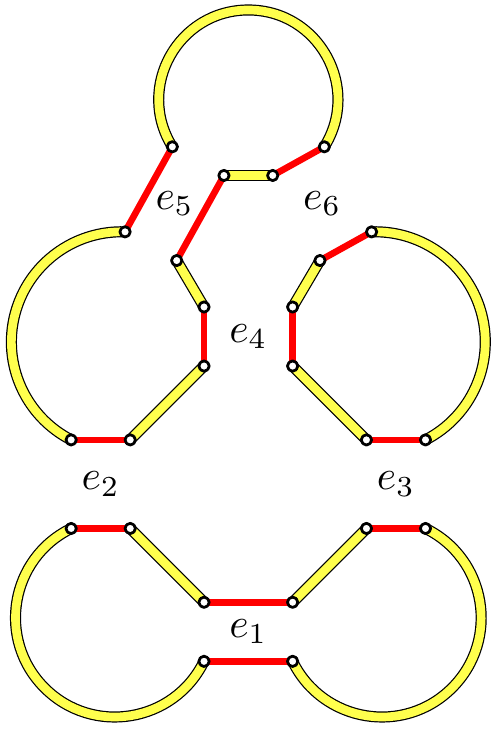}}%
	\kern2cm
	\scalebox{1}{\includegraphics{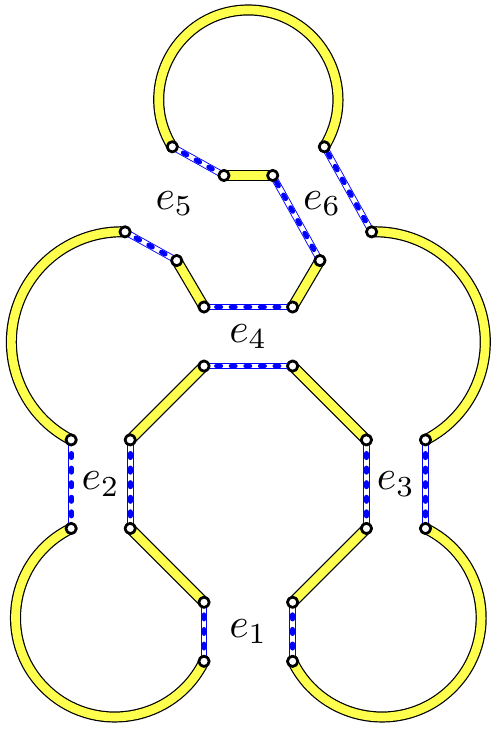}}%
	\hfill
  }
 \caption{Left: new red-yellow bigons (vertices); right: new blue-yellow
	bigons (faces)\label{gem-vf}}
 \end{figure}


 In the figures, the partial duality operations for gems and for the
corner graph seem to be very similar, which is to be expected.
 The corner graph $K$ can be obtained from the coloured dual of the
barycentric subdivision, which is also the embedded gem, by contracting
yellow edges, which become the vertices of $K$.  Then the red edges of
$M$ become $\iR$ in $K$, and the blue edges of $M$ become $\iB$ in $K$.

 A number of properties of partial duals are very easy to handle using
gems.  For example, as mentioned earlier, Chmutov showed that the
orientability of an embedding is unaltered by taking a partial dual. 
This follows from the fact that an embedding is orientable if and only
if its gem is bipartite \cite[pp.~180--181]{Lins82}: this is a fixed
property of the underlying graph of the gem, and is not altered by
swapping colours on red-blue bigons.
 Also, Ellis-Monaghan and Moffatt \cite{EM12, EM} have extensively
investigated the interaction between partial duality and partial Petrie
duality; they call the combined effect \emph{twisted duality}.
 This can be handled using an extension to gems suggested by Lins
\cite{Lins82}.  
 He added a fourth set of edges to a gem, which we call green edges,
adding two diagonals to each red-blue bigon so it becomes a
red-blue-green \emph{trigon}, isomorphic to $K_4$.  Bonnington and
Little \cite{BL} call the result a \emph{jewel}; we show a jewel in
Figure \ref{jewel} (with green indicated by heavy hollow edges).
 Petrie duality is formulated very naturally in the jewel, as swapping
blue and green in red-blue-green trigons: this can obviously be done for
a subset of the edges as well as for the whole edge set.  Therefore,
some aspects of twisted duality are easy to interpret using jewels: for
example, twisted duality provides an action of the \emph{ribbon group}
$S_3^E$ on embeddings with a fixed edge set $E$ \cite[Definition
3.4]{EM12} because we can permute the colours red, blue and green for
each $e \in E$.

 \begin{figure}[tb]
  \lline{%
	\hfill
	\scalebox{1}{\includegraphics{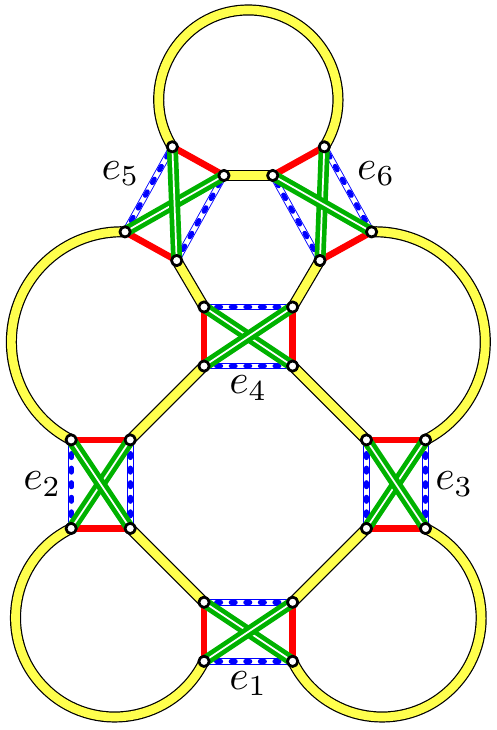}}%
	\hfill
  }
 \caption{The jewel of our earlier example\label{jewel}}
 \end{figure}

 %

 \section{Closed 2-cell embeddings}\label{c2c}

 Closed $2$-cell embeddings are particularly nice because each face and
vertex are incident at most once.
 They can be regarded as embeddings of nonseparable graphs with
facewidth at least $2$.
 The \emph{Strong Embedding Conjecture} \cite{Hagg77} says that
every $2$-connected graph has a closed $2$-cell embedding in some
surface.
 This implies the well-known \emph{Cycle Double Cover Conjecture}
\cite{Seym79,Szek73}, which says that every $2$-edge-connected  graph
has a \emph{cycle double cover}, i.e., a set of cycles in the graph such
that each edge is contained in exactly two of these cycles.


 Being closed $2$-cell seems to be a property that is quite fragile
under taking partial duals.  If we start with a closed $2$-cell graph
embedding, taking the partial dual of any single edge (necessarily a
\emph{link}, or non-loop) creates a loop, and a graph embedding with a
loop and at least one other edge is not closed $2$-cell.  (See
\cite[Subsection 1.7]{Chmu09} or \cite[pp.~32-33]{EM}.  A more
general version of this situation is discussed following Lemma
\ref{global}.)
 However, it is known that the (full geometric) dual of a closed $2$-cell
embedding is always closed $2$-cell.  Therefore, we may wonder whether a
closed $2$-cell embedding has any other closed $2$-cell partial duals
besides its full dual.

 In this section we develop some theory that will allow us to provide
examples of closed $2$-cell embeddings that have nontrivial partial
duals that are closed $2$-cell, and to identify certain situations where the
partial dual will not be closed $2$-cell.
 To do this we provide a necessary and sufficient condition based on the
corner graph for an embedding to be closed $2$-cell.  Some parts of our
condition are unfortunately not easy to check exactly, but stronger
versions of these parts are easy to check, so we can obtain useful
results.
 We also use an argument based on gems to say that some graph embeddings
have no partial duals that are closed $2$-cell.  

 What prevents an embedding from being closed $2$-cell is a \emph{bad
vertex/face pair} $(v,f)$: a vertex $v$ and face $f$ so that $v$ and $f$
are incident more than once.  A \emph{good vertex/face pair} $(v,f)$ is
one where $v$ and $f$ are incident at most once.  The definition of a
bad vertex/face pair is symmetric in faces and vertices, which is why
the full dual of a closed $2$-cell embedding is closed $2$-cell.
 Vertex/face incidences occur precisely at the corners of $G$, so
 these conditions are easy to express using the corner graph $K$: a bad
vertex/face pair occurs if there are some $\iR_v$ and $\iB_f$ that
intersect in two or more vertices of $K$.
 It is also easy to express in terms of 
 the gem: a bad vertex/face pair occurs if
there are a red-yellow bigon and a blue-yellow bigon that share two or
more yellow edges. 

 Therefore, for any particular $G$ and $D$, we can use this corner
graph condition or gem condition, along with Proposition \ref{bd-pd} or
\ref{gem-pd}, to easily determine whether $G^D$ is closed $2$-cell. 
However, we would like to provide some general conclusions.  We do this
by analyzing Proposition \ref{bd-pd} in more detail for corner graphs.

 Suppose we have a graph embedding $G$ and its corner graph $K$ together
in the same surface $\Si$.  We use the notation of the previous section
for band decompositions, including our standard description.  These now
always refer to the corner graph $K$.
Let $D \subseteq E(G)$ and $P = E(G)-D$ be our set of dual and primal
edges, respectively.

 We now classify elements of $G$ and of $K$ based on the division
between primal and dual edges.
 A vertex or face of $G$ is said to be \emph{primal-pure} if all of its
incident edges belong to $P$, and \emph{dual-pure} if all of its
incident edges belong to $D$.  In either case we say it is \emph{pure}. 
 If at least one incident edge belongs to $P$ and at least one to $D$,
then we say a vertex or face of $G$ is \emph{mixed}.
 Recall that vertices of $K$ are corners of $G$, and each $\vk \in V(K)$
has an associated $v = \vkvg(\vk) \in V(G)$, $f = \vkfg(\vk) \in F(G)$,
and $\{e_1,e_2\} = \vkeg(\vk)$ which is an unordered pair of (possibly
equal) edges of $G$.
 We say $\vk$ is \emph{$k$-primal-pure} (\emph{$k$-dual-pure}) if $e_1,
e_2 \in P$ ($e_1, e_2 \in D$) and exactly $k$ of $v$ or $f$ are pure (so
$k = 0$, $1$ or $2$).  We can drop either the $k$ or the primal/dual
distinction, and talk about \emph{$k$-pure} corners, \emph{primal-pure}
or \emph{dual-pure} corners, or just \emph{pure} corners.  A corner that
is not pure is \emph{mixed}, which necessarily means that $e_1 \in D$
and $e_2 \in P$, or vice versa, and that both $v$ and $f$ are mixed.

 We can map every walk $W$ in $K$ to a walk $\wkwg(W)$ in $G$.
 Loosely, we project each edge of $W$ that parallels an edge of $G$
to that edge.  We ignore edges of $W$ that cross edges of $G$; such
edges stay at the same vertex of $G$.
 More formally, suppose $W$ is represented as an alternating sequence of
vertices and edges, $\vk_0, \ek_1, \vk_1, \ek_2, \vk_2, \ldots, \vk_m$
in $K$.  We map this sequence to a walk in $G$ as follows:
 (1) for each $j$, replace vertex $\vk_j$ by $\vkvg(\vk_j)$;
 (2) for each $j$, replace edge $\ek_j$ by $\ekeg(\ek_j)$ if
$\ekcr(\ek_i) = 0$, or delete $\ek_i$ if $\ekcr(\ek_i) = 1$;
 (3) now any consecutive subsequences containing only vertices will all
be repetitions of the same vertex, so replace them by one copy of that
vertex.

 From Proposition \ref{bd-pd} we know that the vertices of $G^D$
correspond to components of $\fR = \bdy \rD$.
 We can apply $\wkwg$ to the cycles $\fR_i$.  The
$\wkwg(\fR_i)$ are exactly the components of boundary walks of faces of
the (possibly non-$2$-cell, possibly disconnected) graph embedding
obtained from $G$ by deleting edges not in $D$.  We keep all vertices of
$G$ even if they have no incident edges from $D$; some of the
$\wkwg(\fR_i)$ may be trivial walks consisting of a single vertex.

 In $G$, let $T_D$ be the subspace of $\Si$ which is the union of $D$
and all dual-pure faces, including all mixed or dual-pure vertices, but
no primal-pure vertices.  We call $T_D$ the \emph{totally dual subspace}
of $\Si$, and its components are \emph{totally dual regions}.
 Similarly, we let $T_P$, the \emph{totally primal subspace}, be the
union of $P$ and all primal-pure faces; its components are \emph{totally
primal regions}.  We note that $T_D \cap T_P$ is precisely the set of
mixed vertices, and that $\Si$ is the disjoint union of $T_D \cup T_P$
and the mixed faces.  The boundary of every mixed face is therefore
contained in $\bdy T_P \cup \bdy T_D$.

 Now each vertex $\fv_i$ of $G^D$, can be classified as one of three
types, depending on where its corresponding cycle $\fR_i$ came from.
 First, suppose $\fR_i = \iR_v$ for some $v \in V(G)$.
 In this case $\wkwg(\fR_i)$ will be a trivial walk $\fv_i$.
 Then $\ir_v$ has no neighbouring faces $\iq_e$ with $e \in D$.  In
other words, there are no edges of $D$ incident with $v$ in $G$, so $v$
is primal-pure.
 We write $\fv_i \in \VPPV$.
 Essentially $v$ and $\fv_i$ are the same vertex, unchanged, and we
write $\fv_i = v$.
 Second, suppose $\fR_i = \iB_f$ for some $f \in F(G)$.
 In this case $\wkwg(\fR_i) = \bdy f$.
 Then all $\iq_e$ adjacent to $\ib_f$ in $K$ must have had $e \in D$. 
In other words, every edge of $f$ in $G$ belongs to $D$, so $f$ is
dual-pure.
 We write $\fv_i \in \VDPF$.
 Essentially the face $f$ has been completely dualized to become the
vertex $\fv_i$ in $G^D$, and we write $\fv_i = f$.
 To discuss the third type, suppose we expand the region $\rD$ to
$\rD^+$ by adding the bands $\bar{\iB_f}$ for all dual-pure faces $f$.
 The components of $\bdy \rD^+$ are exactly the components of $\bdy \rD$
that do not come from dual-pure faces; in other words, they come from
primal-pure vertices, or are of our third type.
 But $\rD^+$ consists of all bands of $K$ that intersect $T_D$, together
with bands $\bar{\ir_v}$ for primal-pure vertices $v = \fv_i$. 
 Therefore, for our third type, $\fR_i$ is a boundary cycle for a
component of $\rD^+$ containing a component $\Phi$ of $T_D$, i.e., a
totally dual region.
 %
 %
 The projected walk $\wkwg(\fR_i)$ in $G$ is one of the closed walks
bounding the region $\Phi$ from the outside (`from the outside' means
there is a slight homotopic shifting of $\wkwg(\fR_i)$ that does not
intersect $\Phi$).
 We write $\fv_i \in \VTDR$.

 We can apply a similar analysis to faces $\ff_j$ of $G^D$: $\ff_j \in
\FPPF$ if it comes from a primal-pure face in $G$; $\ff_j \in \FDPV$ if
it comes from a dual-pure vertex in $G$; $\ff_j \in \FTPR$ if it comes
from a boundary walk of a totally primal region of $G$.

 We illustrate these concepts using Figure \ref{cnr-pd}.
 Cycles $\fR_1$ and $\fR_3$ come from boundary walks of totally dual
regions, so $\fv_1, \fv_3 \in \VTDR$.
 Cycle $\fR_2$ is the same as $\iB_3$, so $\fv_2 = f_3 \in \VDPF$. 
 Cycle $\fR_4$ is the same as $\iR_6$, so $\fv_4 = v_6 \in \VPPV$.
 Cycle $\fB_1$ comes from a boundary walk of a totally primal region, so
$\ff_1 \in \FTPR$.
 Cycle $\fB_2$ is the same as $\iR_5$, so $\ff_2 = v_5 \in \FDPV$. 
 Cycle $\fB_3$ is the same as $\iB_6$, so $\ff_3 = f_6 \in \FPPF$.

 \smallskip
 To eliminate bad vertex/face pairs, therefore, we consider all possible
combinations of the three vertex types and three face types in $G^D$.
 Table \ref{lmg} indicates the way in which we will deal with the nine
resulting cases.
 Here `X' indicates situations that never happen.  If $\fv_i \in \VPPV$
and $\ff_j \in \FDPV$ then $\fR_i$ and $\fB_j$ are both cycles of the
form $\iR_v$ in $K$, and hence are disjoint, so the corresponding vertex
and face of $G^D$ are not incident at all, and cannot form a bad pair. 
Similar reasoning applies when $\fv_i \in \VDPF$ and $\ff_j \in \FPPF$.
 The remaining situations are handled by three conditions that we call
the \emph{local condition}, the \emph{midrange condition} and the
\emph{global condition}, represented in the table by LC, MC and GC,
respectively.

 \begin{table}[tb]
 \centering
 \begin{tabular}{l|r||r|r|r|r}
       \multicolumn{5}{r}{face type\kern25pt}	 \\[2pt]
 \cline{2-5}
       &       & dpv & ppf & tpb \\
 \cline{2-5}
 \noalign{\vskip 2pt}
 \cline{2-5}
\vrule height13pt depth0pt width0pt 
vertex &  ppv & X & LC & MC \\
type   &  dpf & LC & X & MC \\
       &  tdb & MC & MC & GC \\
 \cline{2-5}
 \end{tabular}
 \smallskip
 \caption{Handling bad vertex/face pairs\label{lmg}}
 \end{table}


 The local condition is so called because it deals with fixing local
problems, namely individual bad pairs in $G$.  If there is a bad pair in
$G$, we must make sure it does not yield a bad pair in $G^D$.
 In a graph or graph embedding $G$, we denote by $\cobdy_G v$, or just
$\cobdy v$, the set of edges incident with vertex $v$ (including
loops at $v$).

 \begin{lemma}[Local condition]
 \label{local}
 The following are equivalent:

 \smallskip\noindent
 \textnormal{(i)}
 For every bad vertex/face pair $(v,f)$ in $G$, at least one of $v$
or $f$ is mixed.

 \smallskip\noindent
 \textnormal{(ii)}
 For every bad vertex/face pair $(v,f)$ in $G$, $\cobdy v \cup
\bdy f$ contains both primal and dual edges.

 \smallskip\noindent
 \textnormal{(iii)}
 $G^D$ has no bad vertex/face pair $(\fv_i, \ff_j)$ with $\fv_i
\in \VPPV$ and $\ff_j \in \FPPF$, or $\fv_i \in \VDPF$ and $\ff_j \in
\FDPV$.
 \end{lemma}

 \begin{proof}
 First we show that (i) $\iff$ (ii).  Obviously (i) $\implies$ (ii). 
Suppose (ii) holds.  Since $(v,f)$ is a bad pair, there is a nonempty
set of edges $S$ incident with both $v$ and $f$.  If $S$ contains both
primal and dual edges, both $v$ and $f$ are mixed.  If all edges of $S$
are primal then by (ii) one of $v$ or $f$ must contain a dual edge, and
so is mixed.  A similar argument applies if all edges of $S$ are dual,
and hence (i) holds.

 Before completing the proof, we analyze the pairs discussed in (iii).
 Consider $(\fv_i, \ff_j)$ in $G^D$ with $\fv_i \in \VPPV$ and $\ff_j
\in \FPPF$.
 Then $\fR_i = \iR_v$ and $\fB_j = \iB_f$ for some pair $(v,f)$ in $G$
where $v$ and $f$ are both primal-pure.
 Conversely, every pair $(v,f)$ in $G$ where $v$ and $f$ are both
primal-pure corresponds to a pair $(\fv_i, \ff_j)$ in $G^D$ with $\fv_i
\in \VPPV$ and $\ff_j \in \FPPF$.
 Moreover, $(\fv_i, \ff_j)$ is bad in $G^D$
 $\iff$ $\fR_i$ and $\fB_j$ intersect in two or more vertices
 $\iff$ $\iR_v$ and $\iB_f$ intersect in two or more vertices
 $\iff$ $(v,f)$ is bad in $G$.
 Similarly, there is a one-to-one correspondence between pairs $(\fv_i,
\ff_j)$ in $G^D$ with $\fv_i
\in \VDPF$ and $\ff_j \in \FDPV$ and pairs $(v,f)$ in $G$ where $v$ and
$f$ are both dual-pure, and one pair is bad if and only if the other is.

 The contrapositive of (i) says: (i${}'$) every pair $(v,f)$ in $G$ with
both $v$ and $f$ pure is good.  We prove that (i${}'$) $\iff$ (iii).
 First suppose that (i${}'$) holds.  From above, every pair $(\fv_i,
\ff_j)$ in $G^D$ mentioned in (iii) corresponds to a pair $(v,f)$ in $G$
where both $v$ and $f$ are pure, so that $(v,f)$ is good by (i${}'$). 
From above, $(\fv_i, \ff_j)$ is also good.  So (iii) holds.
 Now suppose that (iii) holds.  From above, this means that every pair
$(v,f)$ in $G$ where $v$ and $f$ are both primal-pure, or $v$ and $f$
are both dual-pure, is good.  But there are no bad pairs $(v,f)$ where
$v$ is primal-pure and $f$ is dual-pure, or vice versa (these are the
`X' situations in Table \ref{lmg}).  Hence, (i${}'$) holds.
 \end{proof}

 Next we prove a result that will be used for our global condition, so
called because it deals with the large-scale structure of the division
between primal and dual edges in $G$.  That is what determines the
elements of $\VTDR$ and $\FTPR$.  Before stating it, we need a formalism
describing how a partition of the edges of a graph embedding into two
parts divides up a surface into two parts.

 The \emph{separation graph} $\sep$ for a set of edges $D$ in a graph
embedding $G$ is obtained from the barycentric subdivision $\bary$ as
follows.  Each triangle $t$ of $\bary$ contains one vertex $u_e \in
U_1$, corresponding to $e \in E(G)$.  We colour $t$ red if $e \in D$,
and blue if $e \notin D$.
 We then take the subgraph of $\bary$ induced by the edges with a red
triangle on one side, and a blue triangle on the other side.
 Since all triangles incident with a vertex of $U_1$ are the same
colour, $V(\sep) \subseteq U_0 \cup U_2$; its elements correspond to
vertices or faces of $G$ that have edges of both $D$ and $E(G)-D$
incident with them.
 Thus, $\sep$ is actually a subgraph of the radial graph.
 
 When $D$ is a set of edges that we intend to dualize, we call $\sep$
the \emph{primal/dual separation graph}.
 Informally, $\sep$ partitions the surface into two (closed) regions,
$\Si_D$ containing dual edges of $\bary$ (union of closures of red faces
of $\bary$) and $\Si_P$ containing primal edges (union of closures of
blue faces of $\bary$); this gives a proper $2$-face-colouring of
$\sep$.  Note that the totally dual and primal subspaces satisfy $T_D
\subseteq \Si_D$ and $T_P \subseteq \Si_P$.  Basically, $\Si_D$ and
$\Si_P$ provide extensions of $T_D$ and $T_P$, respectively, into the
mixed faces of $G$ so that all of $\Si$ is covered.
 Since $\sep$ is a subgraph of the radial graph, each edge of $\sep$
corresponds to a corner, $\vk \in V(K)$, which can be thought of as its
midpoint.  These corners are precisely those with a dual edge on one
side and a primal edge on the other side, namely the mixed vertices of
$K$.

 The global condition is partly expressed in terms of a Petrie dual
(which has also been called the \emph{skew}, as in \cite{BL, Lins82}). 
Besides the representation of Petrie dual discussed earlier (toggling
the signatures of edges in a rotation system/edge signature
representation), the Petrie dual $G\pet$ of a graph embedding $G$ can be
constructed directly from the embedding.  The underlying graph is the
same as for $G$.  The boundaries of the faces of $G\pet$ are determined
by \emph{Petrie walks} in $G$.
 To find a Petrie walk we can apply the usual face-tracing algorithm,
where we stay to one side of an edge as we move along it, but cross over
every edge at its midpoint.  In other words, we can travel along the
edges, maintaining a local orientation, turning left and
right at alternate vertices (so Petrie walks are sometimes called
\emph{zigzag} or \emph{left-right} walks).  We continue until our walk
repeats.  In this way we can find a double-covering of the edges by
closed walks, discard all original faces, and glue a disk onto each
Petrie walk to obtain $G\pet$.

 \begin{lemma}[Global incidences]
 \label{global}
 \textnormal{(a)}
 There is a one-to-one correspondence between $\VTDR \cup \FTPR$ and
faces of $\sep\pet$, the Petrie dual of the primal/dual separation
graph.  The division of $F(\sep\pet)$ into faces corresponding to $\VTDR$
and to $\FTPR$ is a $2$-face-colouring of $\sep\pet$.

 \smallskip\noindent
 \textnormal{(b)}
 Every incidence between a vertex $\fv_i \in \VTDR$ and a face $\ff_j
\in \FTPR$ of $G^D$, at a corner $\vk$, corresponds to one of the
following two situations.

 \begin{itemize}[topsep=\smallskipamount,itemsep=\smallskipamount]
 \item[\textnormal{(i)}]
 Cycles $\fR_i$ and $\fB_j$ meet at a mixed corner $\vk \in V(K)$.  Then
$\fR_i$ and $\fB_j$ cross at $\vk$.  Moreover, this occurs precisely
when the faces of $\sep\pet$ corresponding to $\fv_i$ and $\ff_j$ share
the edge of $\sep$ corresponding to $\vk$.
 \item[\textnormal{(ii)}]
 Cycles $\fR_i$ and $\fB_j$ meet at a $0$-pure corner $\vk \in V(K)$.
 Then $\fR_i$ and $\fB_j$ are tangential (do not cross) at $\vk$.
 \end{itemize}

 \noindent
 In either case, both the vertex $\vkvg(\vk)$ and face $\vkfg(\vk)$ of $G$
associated with $\vk$ are mixed.
 \end{lemma}

 \begin{proof}
 To prove (a) and (b)(i), we need to combine the separation graph and
the corner graph.  Recall that each vertex $\vk$ of the corner graph can
be considered as the midpoint of the edge $\vkvg(\vk) \vkfg(\vk)$ of the
radial graph, and that the separation graph is a subgraph of the radial
graph.  Therefore, we take $\crg$ to be the graph obtained by
subdividing every edge of the radial graph with a vertex of $K$, then
taking the union of the result with $K$.  Thus, $V(\crg) = V(G) \cup U_2
\cup V(K)$.
 In addition to the edges of $K$, for every $\vk \in V(K)$ $\crg$ has a
path $\vkvg(\vk)\, \vk\, \vkfg(\vk)$ of length $2$.
 If we subdivide every edge of $\sep$ to form $\sepsd$, then $\sepsd$ is
a subgraph of $\crg$.  The new vertices created by subdivision are
precisely the mixed vertices of $K$.

 First we consider $\fv_i \in \VTDR$ and the corresponding cycle $\fR_i$
in $K$.  Since $\fv_i \notin \VPPV \cup \VDPF$, $\fR_i$ contains edges
of both $\iB_D$ and $\iR_P$.  So let us begin following $\fR_i$ at a
vertex $\vk_0^-$ with an incident edge of both $\iB_D$ and $\iR_P$ in
$\fR_i$, leaving by the edge of $\iB_D$.
 Remember that $\fR_i$ is a boundary cycle of $\rD$.  
 Note that $\vk_0^-$ is a mixed vertex of $K$.
 We follow edges of $\iB_{f_0} \cap \iQ_D$
where $f_0 = \vkfg(\vk_0^-) \in F(G)$,
passing through dual vertices of $K$, until we reach a mixed vertex
$\vk_0^+$, still with $\vkfg(\vk_0^+) = f_0$.
 At that point we begin following edges of $\iR_{v_0} \cap \iQ_P$, where
$v_0 = \vkvg(\vk_0^+)$, passing through primal vertices of $K$, until we
reach a mixed vertex $\vk_1^-$, still with $\vkvg(\vk_1^-) = v_0$.
 Then we follow edges of $\iB_{f_1} \cap \iQ_D$ for the new face $f_1 =
\vkfg(\vk_1^-) \in F(G)$, and so on.  We continue in this way,
alternating between following paths in $\iB_{f_h} \cap \iQ_D$ for faces
$f_h=f_0, f_1, \ldots$
and paths in $\iR_{v_h} \cap \iQ_P$ for vertices $v_h=v_0, v_1, \ldots$.
 We stop when $\vk_m^- = \vk_0^-$ for some $m \ge 1$, and our
cycle $\fR_i$ has closed.

 We now homotopically shift $\fR_i$ to obtain a new closed walk
$\hssd{\fR_i}$ in $L$, as follows.  Orient $\fR_i$ in the direction we
traversed it above, $\vk_0^- \ldots \vk_0^+ \ldots \vk_1^- \ldots$.  Let
$\fR_i[\vk,\vk']$ be the segment of $\fR_i$ from some $\vk$ to some
$\vk'$ in this direction (if $\vk = \vk'$ this is a single vertex).
 In the segment $\fR_i[\vk_h^-, \vk_h^+] \subseteq \iB_{f_i} \cap
\iQ_D$ every vertex $\vk$ has $\vkfg(\vk) = f_h$, so we can
homotopically shift this to the path $\vk_h^- u_{f_h} \vk_h^+$, for $h =
0, 1, 2, \ldots, m-1$.
 Similarly, in each segment $\fR_i[\vk_h^+, \vk_{h+1}^-] \subseteq
\fR_{v_i} \cap \iQ_P$ (subscripts on $\vk_h^\pm$ interpreted modulo $m$)
every vertex $\vk$ has $\vkvg(\vk) = v_h$, so we can homotopically shift
this to $\vk_h^+ v_h \vk_{h+1}^-$, for $h = 0, 1, 2, \ldots, m-1$.
 Therefore,
 $\hssd{\fR_i} = \vk_0^- u_{f_0} \vk_0^+ v_0 \vk_1^- u_{f_1} \vk_1^+ v_1
	\vk_2^- \ldots v_{m-1} (\vk_m^- = \vk_0^-)$.

 Since each $\vk_h^\pm$ is a mixed vertex of $K$, each path $u_{f_h}
\vk_h^+ v_h$ and each path $v_h \vk_{h+1}^- u_{f_{h+1}}$ is just a
subdivided edge of $\sep$.  Skipping the vertices $\vk^\pm_h$ of degree
$2$ in $\sepsd$, $\hssd{\fR_i}$ becomes a closed walk in $\sep$, namely
 $\hs{\fR_i} = u_{f_0} v_0 u_{f_1} v_1 \ldots v_{m-1} u_{f_0}$.
 We notice that this walk turns to follow $\bdy \Si_D$ at each
$u_{f_h}$, and to follow $\bdy \Si_P$ at each $v_h$.  In other
words, it turns alternately left and right for some local orientation,
and is a Petrie walk in $\sep$, as required.

 In a similar way, there is a Petrie walk $\hs{\fB_j}$ in $\sep$ for
each $\ff_j \in \FTPR$, which turns to follow $\bdy \Si_D$ at each
vertex in $V(G)$, and to follow $\bdy \Si_P$ at each vertex in $U_2$.

 So the $\hs{\fR_i}$ and $\hs{\fB_j}$ form the faces of $\sep\pet$. 
Every mixed corner of $K$ has both some $\fR_i$ for $\fv_i \in \VTDR$
and some $\fB_j$ for $\ff_j \in \FTPR$ passing through it, so every edge
of $\sep$, or $\sep\pet$, belongs to some $\hs{\fR_i}$ and some
$\hs{\fB_j}$.
 Hence, each type of face in $\sep\pet$ covers each edge of $\sep\pet$
exactly once, so we have the required one-to-one correspondence and
$2$-face-colouring.

 \smallskip\noindent
 (b) Each incidence of $\fv_i$ and $\ff_j$ is represented by a vertex
$\vk \in V(K)$ common to $\fR_i$ and $\fB_j$.
 Suppose $\vk$ corresponds to $v \in V(G)$, $f \in F(G)$, and ends of
$e, e' \in E(G)$.  Then the faces around $\vk$ occur in the order
$\ir_v$, $\iq_e$, $\ib_f$, $\iq_{e'}$.

 First suppose $\vk$ is mixed.  Then $e \in D$ and $e' \in P$ (or vice
versa), so $v$ and $f$ are mixed.  The edges incident with $\vk$ in $K$
then belong to $\iR_P$, $\iR_D$, $\iB_D$ and $\iB_P$ in that order (or
its reverse). 
 Now $\fR_i \subseteq \iR_P \cup \iB_D$ and
$\fB_j \subseteq \iR_D \cup \iB_P$, which therefore must cross at
$\vk$.
 Also, our argument for (a), above, shows that if $\vk$ is mixed then
the corresponding faces $\hs{\fR_i}$ and $\hs{\fB_j}$ of $\sep\pet$
share the edge $\vkvg(\vk) u_{\vkfg(\vk)}$ of $\sep\pet$.  Conversely,
if $\hs{\fR_i}$ and $\hs{\fB_j}$ share an edge of $\sep\pet$, or
equivalently of $\sep$, then they both contain the mixed vertex of $K$
with which this edge is subdivided in $\sepsd$.

 So now suppose $\vk$ is pure.  Assume first that $\vk$ is primal-pure. 
Then $e, e' \in P$.  The edges incident with $\vk$ in $K$ are therefore
$\ek_1, \ek_2 \in \iR_v \cap \iQ_P \subseteq \iR_P$ followed by $\ek_3,
\ek_4 \in \iB_f \cap \iQ_P \subseteq \iB_P$ in that order.
 Therefore, $\fR_i \subseteq \iR_P \cup \iB_D$ and $\fB_j \subseteq
\iR_D \cup \iB_P$ do not cross at $\vk$.
 If $v$ is pure, then since $\iR_v$ and $\fR_i$ both contain $\ek_1,
\ek_2 \in \iR_P$, $v$ is primal-pure, and then $\fR_i = \iR_v$ so that
$\fv_i \in \VPPV$, contradicting $\fv_i \in \VTDR$.  So $v$ is mixed.  A
similar argument shows that $f$ is also mixed.  Therefore $\vk$ is
$0$-pure.
 Similar arguments again apply if $\vk$ is dual-pure.

 Therefore, the cases where $\vk$ is mixed or pure correspond exactly to
parts (i) and (ii) as stated.  In both cases $v=\vkvg(\vk)$ and
$f=\vkfg(\vk)$ are mixed.
 \end{proof}

 Condition (b)(i) here is quite nice, being expressible in terms of the
primal/dual separation graph, which reflects global structure. 
 Condition (b)(ii) does not seem to be expressible in such a succinct
way.
 However, it does restrict which corners we need to examine: we do not
need to examine $2$-pure or $1$-pure corners for vertex/face pairs of
this type.

 Note that if $\sep$ has a vertex $\vs$ of degree two, then situation
 (b)(i) occurs on both edges of $\sep\pet$ incident with $\vs$ for the
two faces of $\sep\pet$ that pass through $\vs$, and hence the
corresponding vertex/face pair in $G^D$ is bad.  This happens if there
is a vertex or face of $G$ that is incident exactly once with a region
of $\Si_P$ and exactly once with a region of $\Si_D$.
 For example, dualizing a single link (non-loop) or colink (edge
incident with two distinct faces) always creates a bad pair.

 Finally we prove a result that we will use for our midrange condition,
so called because it deals with the incidence of local structures (faces
or vertices in $G^D$ coming from pure faces or vertices in $G$) with
large-scale structures (faces or vertices in $G^D$ coming from totally
dual or primal regions).
 Like condition (b)(ii) in Lemma \ref{global}, it does not seem to be
expressible in a succinct way, but it does provide some information
about where we should focus our attention.

 \begin{lemma}[Midrange incidences]\label{midrange}
 Suppose we have an incidence in $G^D$ between $\fv_i$ and $\ff_j$
corresponding to a corner $\vk \in V(K)$ where $\fR_i$ and $\fB_j$ meet.
 Let $\vkvg(\vk) = v$ and $\vkfg(\vk) = f$.

 \smallskip\noindent
 \textnormal{(a)}
 If $\fv_i \in \VTDR$ and $\ff_j \in \FDPV$, then $\ff_j$
corresponds to $v$ and $f$ is mixed.

 \smallskip\noindent
 \textnormal{(b)}
 If $\fv_i \in \VTDR$ and $\ff_j \in \FPPF$, then $\ff_j$
corresponds to $f$ and $v$ is mixed.

 \smallskip\noindent
 \textnormal{(c)}
 If $\ff_j \in \FTPR$ and $\fv_i \in \VPPV$, then $\fv_i$
corresponds to $v$ and $f$ is mixed.

 \smallskip\noindent
 \textnormal{(d)}
 If $\ff_j \in \FTPR$ and $\fv_i \in \VDPF$, then $\fv_i$
corresponds to $f$ and $v$ is mixed.

 \smallskip\noindent
 In all of these cases $\vk$ is $1$-pure.
 \end{lemma}

 \begin{proof}
 We prove (a); the others are similar.  Since $\ff_j \in \FDPV$, then
$\fB_j = \iR_{u}$ for some $u \in V(G)$.  But $\fB_j=\iR_u$ shares $\vk$
with $\iR_v$, so $u=v$, and $v$ is dual-pure.  Now if $f$ is pure then
it is also dual-pure, since it contains two edge-ends incident with $v$,
so then $\iB_f = \fR_k$ for some $k$.  But $\fR_k$ shares $\vk$ with
$\fR_i$, so $k=i$ and $\iB_f = \fR_i$, giving $\fv_i \in \VDPF$, a
contradiction.  Hence, $f$ is mixed, and since $v$ is pure, $\vk$ is
$1$-pure.
 \end{proof}

 Essentially, the local incidences occur at $2$-pure corners, the global
incidences occur at mixed or $0$-pure corners, and midrange incidences
occur at $1$-pure corners.  This means that for any given pair of cycles
$\fR_i$ and $\fB_j$, we can restrict our attention only to a certain set
of corners.  If some types of corners do not occur, then it may greatly
simplify the checking of whether $G^D$ is closed $2$-cell or not.

 Now we can state our main theorem, which follows immediately from
Lemmas \ref{local}, \ref{global} and \ref{midrange}.
 Call a pure vertex of $G$ \emph{exposed} if it has more than one
incidence with mixed faces of $G$ (in other words, more than one corner
at that vertex is $1$-pure).
 Similarly, call a pure face of $G$ \emph{exposed} if it has more than
one incidence with mixed vertices of $G$ (in other words, more than one
corner of that face is $1$-pure).

 \begin{theorem}\label{main}
 Suppose $G$ is a graph embedding with at least two edges, and $D
\subseteq E(G)$.  Then the
 partial dual $G^D$ is closed $2$-cell if and only if all of the
following conditions hold.

 \smallskip\noindent
 \textnormal{(LC)}
 For every bad vertex/face pair $(v,f)$ in $G$, at least one of $v$
or $f$ is mixed (or, equivalently, $\cobdy_G v \cup \bdy_G f$ contains
both a primal and a dual edge).

 \smallskip\noindent
 \textnormal{(MC)}
 For every exposed primal-pure (dual-pure) vertex $v$ of $G$, every
boundary walk of a totally primal (dual) region passes through a corner
at $v$ at most once.
 Also, for every exposed dual-pure (primal-pure) face $f$ of $G$, every
boundary walk of a totally primal (dual) region passes through a corner
of $f$ at most once.

 \smallskip\noindent
 \textnormal{(GC)}
 The embedding $\sep\pet$, the Petrie dual of the primal/dual
separation graph of $G$, has no two faces that share two or more edges.
 Also, for each $0$-pure corner, with associated vertex $\fv_i$ and
$\ff_j$ in $G^D$, the faces of $\sep\pet$ associated with $\fv_i$ and
$\ff_j$ do not share an edge in $\sep\pet$.  Nor is there another
$0$-pure corner that is associated with the same $\fv_i$ and $\ff_j$ in
$G^D$.
 \end{theorem}

 One situation where we can check whether an embedding is closed
$2$-cell in a nice way is covered by the following corollary.

 \begin{corollary}\label{simplemain}
 Suppose $G$ is a closed $2$-cell graph embedding with
at least two edges, and $D \subseteq E(G)$.  Suppose that for this $D$
there are no exposed pure vertices or faces, and no $0$-pure corners. 
Then $G^D$ is closed $2$-cell if and only if $\sep\pet$, the Petrie dual
of the primal/dual separation graph, has no two faces that share two or
more edges.
 \end{corollary}

 To convince the reader that this corollary, and indeed our main
theorem, is useful, we now provide an example where Corollary
\ref{simplemain} can be used to demonstrate that a nontrivial partial
dual is closed $2$-cell.
 In Figure \ref{torusgood} we see an example of a closed $2$-cell
embedding on the torus (represented in the standard way as a rectangle,
where we consider the top and bottom sides to be identified and the left
and right sides to be identified).  It is made up of eight diamonds
(copies of $K_4-e$).  We dualize the four diamonds oriented vertically
(dual edges are denoted by wavy curves, as usual).  Then the pure
vertices, pure faces, and totally primal or dual boundary walks are all
similar to those shown in the figure. 
 Every pure vertex touches exactly one mixed face, and every
pure face touches exactly one mixed vertex, so there are no exposed pure
vertices or faces.
 Every corner is mixed (between two diamonds), $2$-pure (four central
corners inside each diamond), or $1$-pure (remaining corners); there are
no $0$-pure corners.
 Therefore, the only thing we need to check is that the Petrie dual of
the primal/dual separation graph does not have two faces sharing two
edges.  Equivalently, we just need to show that boundary walks of a
totally primal region and of a totally dual region never intersect more
than once.  But the boundary walks of totally primal regions run
horizontally across our figure, and those of totally dual regions run
vertically, and so they always meet exactly once.  So we know that the
partial dual $G^D$ constructed in this situation is closed $2$-cell.

 \begin{figure}[tb]
  \lline{%
	\hfill
	\scalebox{1}{\includegraphics{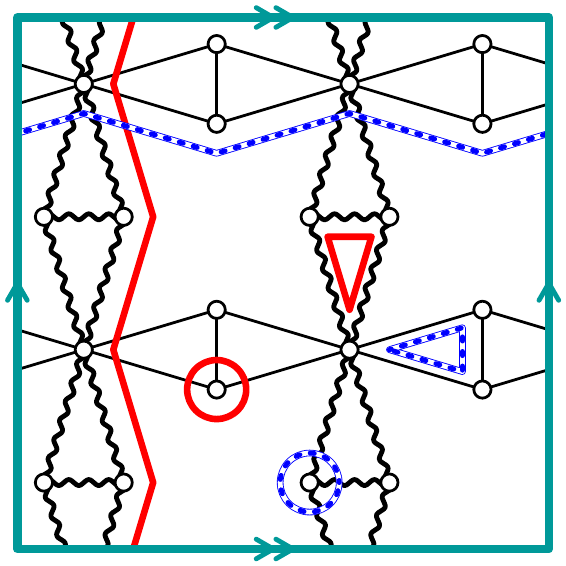}}%
	\hfill
  }
 \caption{A closed $2$-cell nontrivial partial dual
	\label{torusgood}}
 \end{figure}

 We can also provide the reader with another example, also for an
embedding on the torus, where a partial dual is guaranteed \emph{not} to
be closed $2$-cell.
 It is not difficult to construct a toroidal graph embedding with primal
and dual edges so that the primal/dual separation graph $\sep$ forms a
$2$-face-colourable grid, i.e., a standard toroidal embedding of the
cartesian product $C_{2m} \times C_{2n}$ for some $m, n \ge 2$,
represented in the usual rectangular model of the torus by horizontal
and vertical lines. 
 The Petrie walks in the primal/dual separation graph are then
stair-stepping walks that go up and right, or down and right.  The two
Petrie walks sharing an edge go in different directions.  In the usual
representation of the homotopy group of the torus as $\mZ \oplus \mZ$,
the up-and-right walks will have type $(a,b)$, while the down-and-right
walks will have type $(a, -b)$, for some $a, b \ge 1$.  It is well known
that two closed curves of homotopy types $(a, b)$ and $(c, d)$ in the
torus must intersect at least $|ad-bc|$ times, so our walks must
intersect at least $2ab \ge 2$ times.
 If stair-stepping walks going in opposite directions intersect, they
must do so along an edge.  Therefore, any two Petrie walks going in
different directions share at least two edges, and so the partial dual
cannot be closed $2$-cell.


 \smallskip
 We conclude this section by discussing some graph embeddings that have
no partial dual that is closed $2$-cell.  Our condition involves
separability.

 We have mentioned that a graph with a closed $2$-cell embedding must be
nonseparable.  In other words, a separable graph does not have a closed
$2$-cell embedding.  However, a separable graph may have a partial dual
that is closed $2$-cell.
 We can see this by working backwards.  Suppose we take a closed
$2$-cell embedding $G$, and let $D$ be the edges of a spanning tree in
$G$.  
 Then the region $\rD$ is just a fattened version of the tree formed by
$D$, and so has a single boundary component.  Therefore, by Proposition
\ref{bd-pd}, $H=G^D$ has a single vertex, and every edge is a loop, so $H$
is separable (in fact, every edge can be separated from all the other edges).
 But $H^D = G$, which is closed $2$-cell.

  So we need a stronger condition than just separability to guarantee
that a graph embedding $G$ has no partial dual that is closed $2$-cell.
 Our reasoning uses gems, and we need the following standard result
about proper $3$-edge-colourings of cubic graphs.


 \begin{parity}[\cite{Blan46,Desc48}]
 Suppose $S$ is an edge cut in a properly $3$-edge-coloured cubic graph.
Let $S_c$ be the set of all edges in $S$ of colour $c$. 
Then $|S| \cgr |S_c|\ (\mmod 2)$.
 \end{parity}

 \begin{lemma}\label{gem-ec}
 Suppose $M$ is a gem.  Then $M$ is $2$-edge-connected, every $2$-edge
cut in $M$ consists of two edges of the same colour, if $|V(M)| > 4$ and
$M$ has a $2$-edge cut then it has one consisting of two yellow edges,
and if $M$ has no $2$-edge cut then it is cyclically-$4$-edge-connected.
 \end{lemma}

 \begin{proof}
 An edge cut of a single edge cannot satisfy the Parity Lemma for all
colours, so any edge cut must have at least two edges.  Moreover, if
there is a $2$-edge cut, then the Parity Lemma implies that both edges
are the same colour.

 Assume there is a red $2$-edge cut $\{\eg_1, \eg_2\}$ and no yellow
$2$-edge cut.  Suppose $\eg_1$ joins $\vg_1$ and $\vg'_1$.  The red-blue
bigon containing $\eg_1$ must also contain $\eg_2$, and must be a cycle
$(\vg_1 \vg'_1 \vg'_2 \vg_2)$ where $\eg_2$ joins $\vg_2$ and $\vg'_2$. 
 Then the two yellow edges incident with $\vg_1$ and $\vg_2$ form a
$2$-edge cut, unless they are the same edge.  By the same reasoning, the
two yellow edges incident with $\vg'_1$ and $\vg'_2$ form a $2$-edge cut
unless they are the same edge.
 But then $|V(M)|=4$, a contradiction.
 A similar argument applies to a blue $2$-edge cut.

 Suppose $M$ has no $2$-edge cut, but has a $3$-edge cut $S = \{\eg_1,
\eg_2, \eg_3\}$.  Then $\eg_1, \eg_2, \eg_3$ are different colours by
the Parity Lemma.
 Suppose $\eg_1$ is red and $\eg_2$ is blue.  The red-blue bigon
containing $\eg_1$ must use $\eg_2$, which means that $\eg_1$ and
$\eg_2$ share a vertex $\vg$. 
 Let $\eg_4$ be the yellow edge incident with $\vg$.  If $\eg_4 \ne
\eg_3$, then $\{\eg_3, \eg_4\}$ is a $2$-edge cut, which is impossible,
so $\eg_4 = \eg_3$ and all edges of $S$ are incident with the same
vertex.  This is true for all $3$-edge cuts, so $M$ is
cyclically-$4$-edge-connected.
 \end{proof}

 We can interpret $2$-edge cuts in the gem $M$ in terms of the graph
embedding $G$.  Suppose $S = \{\eg_1, \eg_2\}$ is a $2$-edge cut that
separates $M$ into two components $M_1$ and $M_2$.
 By Lemma \ref{gem-ec} both edges of $S$ are the same colour.

 If $S$ is yellow, then there is a red-yellow bigon $\iR_v$ and a
blue-yellow bigon $\iB_f$ that both use $\eg_1$ and $\eg_2$.
 From our discussion at the beginning of this section, this means that
$(v,f)$ is a bad vertex/face pair in $G$.
 Taking the embedding of $M$ corresponding to, and in the same surface
as, the embedding $G$, we can form a simple closed curve $\Psi$
intersecting each of $g_1$ and $g_2$ exactly once, contained in
$\bar{\ir_v} \cup \bar{\ib_f}$ from $M$, and also contained in $f \cup
v$ from $G$.  Then $\Psi$ separates the surface into two components, one
containing $M_1$ and the other containing $M_2$.  Deleting $v$ and $f$
deletes $\Psi$, so it also disconnects the surface.
 We call $(v,f)$ a \emph{separating vertex/face pair}.

 If $S$ is red, then there is a red-yellow bigon $\iR_v$ and a red-blue
bigon $\iQ_e$ that both use $\eg_1$ and $\eg_2$.  In $G$ this means that
the edge $e$ is a loop at $v$.  In a manner similar to the previous
paragraph, we can construct a simple closed curve $\Psi$ contained in
$\bar{\ir_v} \cup \bar{\iq_e}$, and also passing through $v$, which
separates $\Si$.  In this case $\Psi$ is homotopic (with base point $v$)
to $e$, and so $e$ is a separating simple closed curve.  We call $e$ a
\emph{separating loop}.

 If $S$ is blue, then we have the dual situation to when $S$ is red.  In
this case the edge $e^*$ in $G^*$ dual to $e$ is a separating loop, so
we say $e$ is a \emph{separating coloop}: $e$ is an edge with the same
face on both sides that is also a cutedge in $G$.

 Putting these together with Lemma \ref{gem-ec} we obtain the following
sufficient condition for a graph embedding to have no closed $2$-cell
partial duals.

 \begin{theorem}
 If a graph embedding $G$ with at least two edges has a separating
vertex/face pair, a separating loop, or a separating coloop, then every
partial dual of $G$ has a separating vertex/face pair.  Hence, no
partial dual of $G$ is closed $2$-cell.
 \end{theorem}

 \begin{proof}
 The hypothesis lists the situations under which the gem of $G$ has a
$2$-edge-cut.  Then by Lemma \ref{gem-ec} it has a yellow $2$-edge cut,
which remains a yellow $2$-edge cut in the gem of every partial dual. 
 Hence every partial dual has a separating vertex/face pair, which is a
bad pair, so the partial dual is not closed $2$-cell.
 \end{proof}

 \begin{question}
 For a graph embedding to have a closed $2$-cell partial dual, we have
seen that it is necessary for the gem to have no $2$-edge-cut.  It is
natural to ask whether this is sufficient.  In other words, if $G$ is a
graph embedding with no separating vertex/face pair, does there exist a
closed $2$-cell partial dual of $G$?
 \end{question}

 To conclude, we have developed some techniques in this paper which
allow us to give some conditions under which a partial dual is or is not
closed $2$-cell.  One goal would natually be to try to find ways to use
this to construct closed $2$-cell embeddings of some classes of graph,
to address the Strong Embedding and Cycle Double Cover Conjectures. 
 The difficulty is that partial duality does not preserve the underlying
graph (unlike, for example, partial Petrie duality).  
 It therefore seems difficult to incorporate partial duality into an
attack on these problems.

 %
 %
 %
 %

 \section*{Acknowledgement}

 The authors thank both referees for helpful comments that improved the
presentation of the paper.


\end{document}